\documentclass[a4paper,10pt]{amsart}
\usepackage[utf8x]{inputenc}
\usepackage{amsthm,amssymb, amsmath}
\usepackage{graphicx}
\usepackage{stmaryrd}
\usepackage[margin=5pt,font=small]{caption}

\newtheorem{theorem}{Theorem}

\newtheorem{lemma}[theorem]{Lemma}
\newtheorem{proposition}[theorem]{Proposition}

\newtheorem{definition}[theorem]{Definition}
\newtheorem{remark}[theorem]{Remark}

\newcommand{\QQ}{\mathbb{Q}}

\newcommand{\NN}{\mathbb{N}}

\newcommand{\EE}{\mathcal{E}}
\newcommand{\BB}{\mathcal{B}}

\newcommand{\QUT}{\mathbb{Q}_{UT}}

\title{Tuning and plateaux for the entropy of $\alpha$-continued fractions}
\author{Carlo Carminati, Giulio Tiozzo}
\address[Carlo Carminati]{Dipartimento di Matematica \\
Universit\`a di Pisa \\ Largo Bruno Pontecorvo 5, I-56127, Italy}
\email[Carlo Carminati]{carminat@dm.unipi.it}
\address[Giulio Tiozzo]{Department of Mathematics \\
Harvard University\\
One Oxford Street Cambridge MA 02138 USA}
\email[Giulio Tiozzo]{tiozzo@math.harvard.edu}

\begin{document}

\begin{abstract}
The entropy $h(T_\alpha)$ of $\alpha$-continued fraction
transformations is known to be locally monotone outside a closed, 
totally disconnected set $\EE$.  We will exploit the explicit description
of the fractal structure of $\EE$ to
investigate the self-similarities displayed by the graph of the
function $\alpha \mapsto h(T_\alpha)$. 
Finally, we completely characterize the plateaux occurring in this graph, 
and classify the local monotonic behaviour.
\end{abstract}

\maketitle

\section{Introduction}

It is a well-known fact that the continued fraction expansion of a real number 
can be analyzed in terms of the dynamics of the interval map $G(x) := \left\{ \frac{1}{x} \right\}$, 
known as the \emph{Gauss map}.  
A generalization of this map is given by the family of \emph{$\alpha$-continued fraction transformations} $T_\alpha$, 
which will be the object of study of the present paper.
For each $\alpha \in [0,1]$, the map $T_\alpha:[\alpha-1, \alpha] \to [\alpha-1, \alpha]$ 
is defined as $T_\alpha(0) = 0$ and, for $x \neq 0$,
$$T_\alpha(x) := \frac{1}{|x|} - c_{\alpha, x}$$ where $c_{\alpha, x}
    = \left\lfloor \frac{1}{|x|} + 1 - \alpha \right\rfloor$ is a
    positive integer. 
Each of these maps is associated to a different continued fraction
expansion algorithm, and the family $T_\alpha$ interpolates between maps associated to 
well-known expansions: $T_1 = G$ is the usual Gauss map which generates regular
continued fractions, while $T_{1/2}$ is associated to the \emph{continued
fraction to the nearest integer}, and $T_0$ generates the \emph{by-excess continued fraction} expansion.
For more about $\alpha$-continued fraction expansions, their metric properties and their relations with 
other continued fraction expansions we refer to \cite{N}, \cite{Sc}, \cite{IK}. This family has also been 
studied in relation to the Brjuno function \cite{MMY}, \cite{MCM}.

Every $T_\alpha$ has infinitely many branches, and, for $\alpha>0$,
all branches are expansive and $T_\alpha$ admits an invariant
probability measure absolutely continuous with respect to Lebesgue
measure. Hence, each $T_\alpha$ has a well-defined metric entropy
$h(\alpha)$: the metric entropy of the map $T_\alpha$ is proportional to the 
speed of convergence of the corresponding expansion algorithm (known as \emph{$\alpha$-euclidean algorithm}) \cite{BDV}, 
and to the exponential growth rate of the partial quotients in the $\alpha$-expansion of typical
values \cite{NN}.


Nakada \cite{N}, who first investigated the properties of this family
of continued fraction algorithms, gave an explicit formula for $h(\alpha)$
for $\frac{1}{2} \leq \alpha \leq 1$, from which it is evident
that entropy displays a phase transition phenomenon when the parameter 
equals the golden mean $g:=
\frac{\sqrt{5}-1}{2}$
 (see also Figure \ref{tuningdiagram}, left):
\begin{equation}\label{eq:N81}
h(\alpha)= \left\{
\begin{array}{ll}
 \frac{\pi^2}{6 \log(1 + \alpha)} &
\textup{for } \frac{\sqrt{5}-1}{2} < \alpha \leq 1 \\
                         \frac{\pi^2}{6 \log \frac{\sqrt{5}+1}{2}} & \textup{for
}\frac{1}{2} \leq \alpha \leq \frac{\sqrt{5}-1}{2}
\end{array} 
\right.
\end{equation}

Several authors have studied the behaviour of the metric entropy of
$T_\alpha$ as a function of the parameter $\alpha$ (\cite{C},
\cite{LM}, \cite{NN}, \cite{KSS});
in particular Luzzi and Marmi \cite{LM} first produced numerical evidence that the
entropy is continuous, although it displays many more (even if less
evident) phase transition points and it is not monotone on the
interval $[0,1/2]$.  Subsequently, Nakada and Natsui \cite{NN} 
identified a dynamical condition that forces the entropy to be, 
at least locally, monotone: indeed, they noted that for some parameters $\alpha$,
the orbits under $T_\alpha$ of $\alpha$ and $\alpha-1$ collide after a number of steps, 
i.e. there exist $N, M$ such that: 
\begin{equation} \label{mat}
T_\alpha^{N+1}(\alpha) = T_\alpha^{M+1}(\alpha-1)
\end{equation}
and they proved that, whenever the {\it matching condition} \eqref{mat} 
holds, $h(\alpha)$ is monotone on a neighbourhood of $\alpha$.
They also showed that $h$ has {\em mixed monotonic behaviour}
near the origin: namely, for every $\delta>0$, in the interval
$(0,\delta)$ there are intervals on which $h(\alpha)$ is monotone, others on
which $h(\alpha)$ is increasing and others on which $h(\alpha)$ is decreasing.

In \cite{CT} it is proven that the set of parameters for which \eqref{mat} holds 
actually has full measure in parameter space. Moreover, such 
a set is the union of countably many open intervals, called \emph{maximal 
quadratic intervals}. Each maximal quadratic interval $I_r$ is labeled 
 by a rational 
number $r$ and can be thought of as a stability domain in parameter space: 
indeed, the number of steps $M, N$ it takes for the 
orbits to collide is the same for each $\alpha \in I_r$, and even the symbolic orbit of $\alpha$ and $\alpha-1$ up to the collision
 is fixed (compare to \emph{mode-locking} phenomena in the theory of circle maps).
For this reason, the complement of the union of all $I_r$ is called 
the \emph{bifurcation set} or \emph{exceptional set} $\EE$.

Numerical experiments \cite{LM}, \cite{CMPT} show the entropy function 
$h(\alpha)$ displays self-similar features: the main goal of this paper is 
to prove such self-similar structure by exploiting the self-similarity of 
the bifurcation set $\EE$. 

The way to study the self-similar structure was suggested to us by the unexpected 
isomorphism between $\EE$ and the real slice of the boundary of the Mandelbrot set
\cite{BCIT}. In the family of quadratic polynomials, Douady and Hubbard \cite{DH}
described the small copies of the Mandelbrot set which appear inside the large
Mandelbrot set as images of \emph{tuning operators}: we define a 
similar family of operators using the dictionary of \cite{BCIT}. 
(We refer the reader to the Appendix for more about this correspondence, even 
though knowledge of the complex-dynamical picture is strictly speaking not
necessary in the rest of the paper.)

Our construction is the following: we associate, to each rational number $r$
indexing a maximal interval, a \emph{tuning map} $\tau_r$
from the whole parameter space of $\alpha$-continued fraction transformations to 
a subset $W_r$, called \emph{tuning window}. 
Note that $\tau_r$ also maps the bifurcation set $\mathcal{E}$ into itself. 
A tuning window $W_r$ is called \emph{neutral} if the alternating sum of the partial quotients of $r$ is zero.
Let us define a \emph{plateau} of a real-valued function as a maximal, connected open set where the function is constant.

   


\begin{theorem} \label{plateaux}
The function $h$ is constant on every neutral tuning window $W_r$, and
every plateau of $h$ is the interior of some neutral tuning window
$W_r$.
\end{theorem}
Even more precisely, we will characterize the set of rational numbers
$r$ such that the interior of $W_r$ is a plateau (see Theorem \ref{last}).  A particular case of the
theorem is the following recent result \cite{KSS}:
$$h(\alpha) = \frac{\pi^2}{6 \log(1+g)} \qquad \forall \alpha \in [g^2, g],$$
and $(g^2, g)$ is a plateau (i.e. $h$ is not constant on $[t,g]$ for any $t<g^2$).

On non-neutral tuning windows, instead, entropy is non-constant and 
$h$ reproduces, on a smaller scale, its behaviour on the whole parameter space $[0,1]$.

\begin{theorem} \label{main}
If $h$ is increasing on
 a maximal interval  $I_r$, then the monotonicity of $h$ on the tuning window $W_r$
reproduces the behaviour on the interval $[0, 1]$, but with reversed sign:
more precisely, if $I_p$ is another maximal interval, then
\begin{enumerate}
\item
$h$ is increasing on $I_{\tau_r(p)}$ iff it is decreasing on $I_p$;
\item
$h$ is decreasing on $I_{\tau_r(p)}$ iff it is increasing on $I_p$;
\item
$h$ is constant on $I_{\tau_r(p)}$ iff it is constant on $I_p$.
\end{enumerate}
If, instead, $h$ is 
decreasing on $I_r$, then the monotonicity of $I_p$ and $I_{\tau_r(p)}$ is the same.
\end{theorem}

\begin{figure}[ht] 
\centering
\includegraphics[scale=0.2]{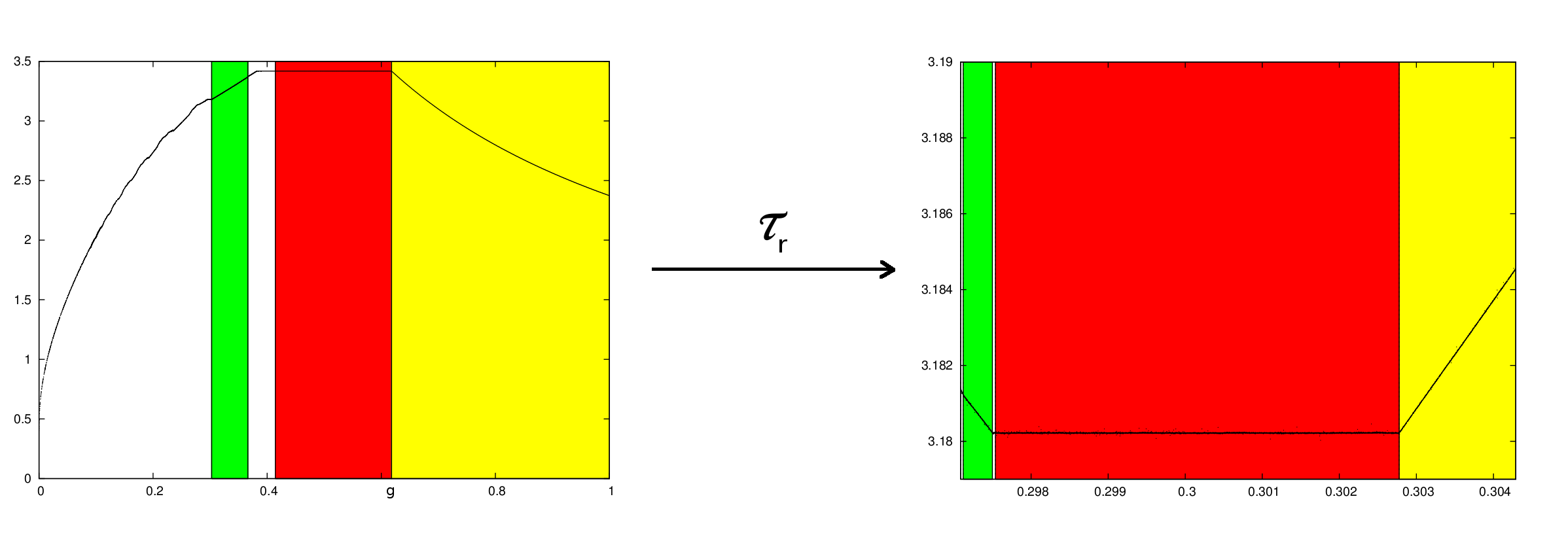}
\caption{An illustration of Theorem \ref{main} is given in the picture: on the left, you see the whole parameter space $[0,1]$, and the graph of $h$. 
The colored strips correspond to three maximal intervals. On the right, 
 $x$ ranges on the tuning window $W_{1/3} = [\frac{5 - \sqrt{3}}{22}, \frac{\sqrt{3}-1}{2})$ 
relative to $r = 1/3$. 
Maximal intervals on the left are mapped via $\tau_r$ to maximal intervals of the same color on the right. As prescribed by Theorem \ref{main}, 
the monotonicity of $h$ on corresponding intervals is reversed. Note that in the white strips (even if barely visible on the right) there are 
infinitely many maximal quadratic intervals.
}
\label{tuningdiagram}
\end{figure}

As a consequence, we can also completely classify the local monotonic behaviour of
 the entropy function $\alpha \mapsto h(\alpha)$:

\begin{theorem} \label{classmon}
Let $\alpha$ be a parameter in the parameter space of $\alpha$-continued fractions. Then:
\begin{enumerate}
 \item if $\alpha \notin \mathcal{E}$, then $h$ is monotone on a neighbourhood of $\alpha$;
\item if $\alpha \in \mathcal{E}$, then either
\begin{itemize}
 \item[(i)] 
$\alpha$ is a \emph{phase transition}: $h$ is constant on the left of $\alpha$ and strictly monotone (increasing or decreasing) on the right of $\alpha$;
\item[(ii)]
$\alpha$ lies in the interior of a 
\emph{neutral tuning window}: then $h$ is constant on a neighbourhood of $\alpha$;
\item[(iii)]
otherwise, $h$ has \emph{mixed monotonic behaviour} at $\alpha$, i.e. in every neighbourhood of $\alpha$ there are infinitely many 
intervals on which $h$ is increasing, infinitely many on which it is decreasing and infinitely many on which it is constant.
\end{itemize}
\end{enumerate}
\end{theorem}

Note that all cases occur for infinitely many parameters: more precisely, 1. occurs for a set of 
parameters of full Lebesgue measure; 2.(i) for a countable set of parameters; 2.(ii) for a set of parameters whose Hausdorff 
dimension is positive, but smaller than $\frac{1}{2}$; 
2.(iii) for a set of parameters of Hausdorff dimension $1$. 
Note also that all phase transitions are of the form $\alpha = \tau_r(g)$, i.e. they are 
tuned images of the phase transition at $\alpha = g$ which is described by formula \eqref{eq:N81}. The largest parameter for which 2.(iii) occurs is indeed 
$\alpha =  g^2$, which is the left endpoint of the neutral tuning window $W_{1/2}$. Moreover, there is an explicit algorithm to decide, 
whenever $\alpha$ is a quadratic irrational, which of these cases occurs.

The structure of the paper is as follows. In section \ref{backg}, we introduce basic notation and definitions 
about continued fractions, and in section \ref{match} we recall the construction and results from \cite{CT} which 
are relevant in this paper. We then define the tuning operators and establish their basic properties (section \ref{sec:TW}),
and discuss the behaviour of tuning with respect to monotonicity of entropy, thus proving Theorem \ref{main} (section \ref{sec:mono}).
In section \ref{hoelder} we discuss untuned and dominant parameters, and use them to prove the characterization of plateaux (Theorem \ref{plateaux}
 above, and Theorem \ref{last}). Finally, section \ref{sec:class} is devoted to the proof of Theorem \ref{classmon}.

\section{Background and definitions} \label{backg}

\subsection{Continued fractions.}

The continued fraction expansion of a number 
$$x = \frac{1}{a_1+ \frac{1}{a_2 + \dots}}$$
will be denoted by $x = [0; a_1, a_2, \dots]$, and the $n^{th}$ convergent of $x$ will be denoted by
$\frac{p_n}{q_n}:=[0;a_1,...,a_n]$. Often we will also use the 
compact notation $x = [0; S]$ where $S = (a_1, a_2, \dots)$ is the
(finite or infinite) string of partial quotients of $x$. 

If $S$ is a finite string, its length will be denoted by $|S|$. 
A string $A$ is a \emph{prefix} of $S$ if there exists a (possibly empty) string $B$ such that $S = AB$; 
$A$ is a \emph{suffix} of $S$ if there exists a (possibly empty) string $B$ such that $S = BA$; $A$ is 
a \emph{proper suffix} of $S$ if there exists a non-empty string $B$ such that $S = BA$.

\begin{definition}
Let $S = (s_1, \dots, s_n)$, $T = (t_1, \dots, t_n)$ be two strings of positive integers of equal length. We
say that $S < T$ if there exists $0 \leq k < n$ such that 
$$s_i = t_i \qquad \forall 1 \leq i \leq k
 \qquad  \qquad \mbox{ and }  \qquad  \qquad
\left\{ \begin{array}{cc} s_{k+1} < t_{k+1} & \textup{if }k\textup{ is odd} \\
                            s_{k+1} > t_{k+1} & \textup{if }k\textup{ is even}              
          \end{array} \right.
$$ 
\end{definition}

This is a total order on the set of strings of given length, and it is defined so that $S < T$ iff $[0; S] < [0; T]$. 
As an example, $(2, 1) < (1, 1) < (1, 2)$.
Moreover, this order can be extended to a partial order on the set of all finite strings of positive integers in the following way:
\begin{definition}
If $S = (s_1, \dots, s_n)$, $T = (t_1, \dots, t_m)$ are strings of finite (not necessarily equal) length, then 
we define $S << T$ if there exists $0 \leq k < \min\{n, m\}$ such that
$$s_i = t_i \qquad \forall 1 \leq i \leq k 
\qquad \qquad \mbox{ and } \qquad  \qquad
\left\{ \begin{array}{cc} s_{k+1} < t_{k+1} & \textup{if }k\textup{ is odd} \\
                            s_{k+1} > t_{k+1} & \textup{if }k\textup{ is even}              
          \end{array} \right.
$$ 
\end{definition}

As an example, $(2, 1) << (1)$, and $(2, 1, 2) << (2, 2)$.
This order has the following properties:
\begin{enumerate}
\item if $|S| = |T|$, then $S < T$ if and only if $S << T$;
\item if $X, Y$ are infinite strings and $S << T$, then $[0; SX] < [0; TY]$;
\item if $A\leq B$ and $B<<C$, then $A<<C$. 
\end{enumerate}

\subsection{Fractal sets defined by continued fractions.}

We can define an action of the semigroup of finite strings
(with the operation of concatenation) on the unit interval. Indeed,
for each $S$, we denote by $S \cdot x$ the number obtained by
appending the string $S$ at the beginning of the continued fraction
expansion of $x$; by convention the empty string corresponds to the identity.  

We shall also use the notation $f_S(x) := S \cdot
x$; let us point out that the Gauss map $G(x) := \left\{\frac{1}{x} \right\}$ 
acts as a shift on continued fraction expansions, hence $f_S$ is a right inverse of $G^{|S|}$ ($G^{|S|}\circ
f_S(x)=x$).
It is easy to check that concatenation of
strings corresponds to composition  $ (ST)\cdot x = S \cdot(T\cdot x)$; moreover, the map $f_S$ is increasing
if $|S|$ is even, decreasing if it is odd.
It is not hard to see that $f_S$ is given by the formula
\begin{equation}\label{eq:fractrans}
f_S(x)=\frac{p_{n-1}x+p_n}{q_{n-1}x+q_n}
\end{equation}
where $\frac{p_n}{q_n} = [0; a_1, \dots, a_n]$ and $\frac{p_{n-1}}{q_{n-1}} = [0; a_1, \dots, a_{n-1}]$.
The map $f_S$ is a contraction of the unit interval: indeed, by taking the derivative 
in the previous formula and using the relation $q_n p_{n-1} - p_n q_{n-1}=(-1)^n$ (see \cite{IK}), 
$f_S'(x) = \frac{(-1)^n}{(q_{n-1}x + q_n)^2}$, hence

\begin{equation} \label{contraction}
\frac{1}{4 q(S)^2} \leq |f'_S(x)| \leq \frac{1}{q(S)^2} \qquad \qquad \forall x \in [0,1]
\end{equation}
where $q(S) = q_n$ is the denominator of the rational number whose c.f. expansion is $S$. 


A common way of defining Cantor sets via continued fraction expansions
is the following:
\begin{definition} \label{RCS}
Given a finite set $\mathcal{A}$ of finite strings of positive integers, the \emph{regular Cantor set} defined by $\mathcal{A}$ is the set
$$K(\mathcal{A}) := \{ x = [0; W_1, W_2, \dots ] \ : \ W_i \in \mathcal{A} \ \forall i \geq 1\}$$
\end{definition}
For instance, the case when the alphabet $\mathcal{A}$ consists of
strings with a single digit gives rise to sets of continued fractions with {\em restricted digits} \cite{Hen}. 

An important geometric invariant associated to a fractal subset $K$ of the real line is its {\em Hausdorff dimension}
$\textup{H.dim }K$. In particular, a regular Cantor set is generated by an {\em iterated function system}, and its dimension 
can be estimated in a standard way (for basic properties about Hausdorff dimension we refer to Falconer's book \cite{F}, in particular Chapter 9). 

Indeed, if the alphabet $\mathcal{A} = \{S_1, \dots, S_k\}$ is not redundant (in the sense that 
no $S_i$ is prefix of any $S_j$ with $i \neq j$), the dimension of $K(\mathcal{A})$ is bounded in 
terms of the smallest and largest contraction factors of the maps $f_W$ (\cite{F}, Proposition 9.6):
\begin{equation} \label{dimbounds}
\frac{\log N}{- \log m_1} \leq \textup{H.dim } K(\mathcal{A}) \leq \frac{\log N}{-\log m_2}
\end{equation}
where $m_1 := \inf_{\stackrel{W \in
  \mathcal{A}}{x \in [0, 1]}}  |f_W'(x)|$, $m_2 := \sup_{\stackrel{W \in
  \mathcal{A}}{x \in [0, 1]}} |f_W'(x)|$, and 
$N$ is the cardinality of $\mathcal{A}$.

\section{Matching intervals} \label{match}

Let us now briefly recall the main construction of \cite{CT}, which will be essential in the following.

Each irrational number has a unique infinite continued fraction
expansions, while every rational number has exactly two finite
expansions.  In this way, one can associate to every rational $r \in
\mathbb{Q} \cap (0,1)$ two finite strings of positive integers: let
$S_0$ be the string of even length, and $S_1$ be the one of odd
length.  For instance, since $3/10=[0;3,3]=[0;3,2,1]$, the two strings
associated to $3/10$ will be $S_0=(3,3)$ and $S_1=(3,2,1)$.  Let us
remark that, if $r=[0;S_0]=[0;S_1]$ then $S_0>>S_1$.

Now, for each $r \in \mathbb{Q} \cap (0,1)$ we define the \emph{quadratic interval} associated
to $r$ as the open interval 
$$I_r:=(\alpha_1,\alpha_0)$$
whose endpoints are the two quadratic irrationals $\alpha_0=[0;\overline{S_0}]$ and
$\alpha_1=[0;\overline{S_1}]$. 
 It is easy to check that $r$ always belongs to $I_r$, and it is the unique element of $\mathbb{Q} \cap I_r$
which is a convergent of both endpoints of $I_r$. In fact $r$ is the rational with minimal denominator in
$I_r$, and it will be called the \emph{pseudocenter} of $I_r$.
Let us define the \emph{bifurcation set} (or \emph{exceptional set} in the terminology of \cite{CT}) as

\begin{equation} \label{Edef}
\mathcal{E} := [0,g]\setminus \bigcup_{r \in (0, 1)\cap \mathbb{Q}} I_r
\end{equation}

The intervals $I_r$ will often overlap; however, by (\cite{CT}, Prop. 2.4 and Lemma 2.6), 
the connected components of $[0,g]\setminus \EE$ are themselves quadratic intervals, called \emph{maximal quadratic intervals}. 
That is to say, every quadratic interval is contained in a unique maximal quadratic interval, and two distinct 
maximal quadratic intervals do not intersect. This way, the set of pseudocenters 
of maximal quadratic intervals is a canonically defined subset of $\mathbb{Q} \cap (0,1)$ and will be 
denoted by
$$\QQ_E := \{ r \in (0,1) \ : \ I_r \textup{ is maximal} \}$$
We shall sometimes refer to $\QQ_E$ as the set of {\em extremal rational values}; this is motivated by the  following  characterization of $\QQ_E$:
\begin{proposition}[\cite{CT}, Proposition 4.5] \label{stringlemma}
A rational number $r = [0; S]$ belongs to $\QQ_E$ if and only if, for any splitting $S=AB$ of $S$ into two 
strings $A$, $B$ of positive length, either
$$AB < BA$$ 
or $A = B$ with $|A|$ odd.
\end{proposition}
Using this criterion, for instance, one can check that $[0; 3, 2]$ belongs to $\mathbb{Q}_E$
(because $(3, 2) < (2, 3)$), and so does $[0; 3, 3]$,  
while $[0; 2, 2, 1, 1]$ does not (indeed, $(2, 1, 1, 2) < (2, 2, 1, 1)$).
Related to the criterion is the following characterization of $\mathcal{E}$ in terms of orbits of the Gauss map $G$:
 
\begin{proposition}[\cite{BCIT}, Lemma 3.3] 
$$\EE= \{ x\in[0,1] \ : G^k(x) \geq x \ \  \forall k \in \mathbb{N} \}.$$
\end{proposition}

For $t\in (0,1)$ fixed, let us also define the closed set
$$\BB(t) := \{ x\in[0,1] \ : G^k(x) \geq t \ \ \forall k \in \mathbb{N} \}.$$ 
To get a rough idea of the meaning of the sets $\BB(t)$ let us mention
that for $t=1/(N+1)$ one gets the set of values whose continued
fraction expansion is infinite and contains only the digits
$\{1,...,N\}$ as partial quotients.
A simple relation follows from the definitions:
\begin{remark} \label{inclusion}
For each $t \in [0, 1]$, $\mathcal{E} \cap [t, 1] \subseteq \mathcal{B}(t)$.
\end{remark}
A thorough study of the sets $\mathcal{B}(t)$ and their interesting
connection with $\mathcal{E}$ is contained in \cite{CT2}. Note that,
from Remark \ref{inclusion} and ergodicity of the Gauss map, it
follows that the Lebesgue measure of $\mathcal{E}$ is zero.

\subsection{Maximal intervals and matching.}

Let us now relate the previous construction to the dynamics of $\alpha$-continued fractions. 
The main result of \cite{CT} is that for all parameters $\alpha$ belonging to a 
maximal quadratic interval $I_r$, the orbits of $\alpha$ and $\alpha-1$ under the $\alpha$-continued 
fraction transformation $T_\alpha$ coincide after a finite number of steps, and this number of steps depends only on the 
usual continued fraction expansion of the pseudocenter $r$:


\begin{theorem}[\cite{CT}, Thm 3.1] \label{matching}
Let $I_r$ be a maximal quadratic interval, and $r = [0; a_1, \dots, a_n]$ with $n$ even. Let 
\begin{equation}\label{eq:index}
N = \sum_{i \textup{ even}} a_i \qquad M = \sum_{i \textup{ odd}} a_i
\end{equation}
Then for all $\alpha \in I_r$, 
\begin{equation} \label{eq:matching}
T_\alpha^{N+1}(\alpha) = T_\alpha^{M+1}(\alpha-1) 
\end{equation}
\end{theorem}

Equation \eqref{eq:matching} is called {\em matching condition}. Notice that
$N$ and $M$ are the same for all $\alpha$ which belong to the open interval $I_r$. 
Indeed, even more is true, namely the symbolic orbits of $\alpha$ and $\alpha-1$ up to steps 
respectively
$N$ and $M$ are constant 
over all the interval $I_r$ (\cite{CT}, Lemma 3.7). 
Thus we can regard each maximal quadratic
interval as a stability domain for the family of $\alpha$-continued
fraction transformations, and the complement $\mathcal{E}$ as the
bifurcation locus.

One remarkable phenomenon, 
which was first discovered by Nakada and Natsui (\cite{NN}, Thm. 2), 
is that the matching condition locally determines the monotonic behaviour 
of $h(\alpha)$: 


\begin{proposition}[\cite{CT}, Proposition 3.8] \label{Nakada}
Let $I_r$ be a maximal quadratic interval, and let $N,M$ be as in
Theorem \ref{matching}. Then:
\begin{enumerate}
\item if $N < M$, the entropy $h(\alpha)$ is increasing for $\alpha \in I_r$; 
\item if $N = M$ it is constant on $I_r$;
\item if $N > M$ it is decreasing on $I_r$.
\end{enumerate}
\end{proposition}

%

\section{Tuning}\label{sec:TW}

Let us now define \emph{tuning operators} acting on parameter space, inspired by the dictionary with complex dynamics (see the Appendix).
We will then see how such operators are responsible for the self-similar structure of the entropy.

\subsection{Tuning windows}

Let $r\in
\QQ_E$ be the pseudocenter of the maximal interval $I_r=(\alpha_1,
\alpha_0)$; if $r=[0;S_0]=[0;S_1]$ are the even and odd expansions of
$r$, then $\alpha_i=[0;\overline{S_i}]$ ($i=0,1$). Let us also set
$\omega:=[0;S_1\overline{S_0}]$ and define the tuning window generated
by $r$ as the interval
$$W_r:= [\omega, \alpha_0).$$ 
The value $\alpha_0$ will be called the \emph{root} of the tuning window.
For instance, if $r = \frac{1}{2} = [0; 2] = [0; 1, 1]$, then $\omega
= [0; 2, \overline{1}] = g^2$ and the root $\alpha_0 = [0; \overline{1}] = g$.

The following proposition describes in more detail the structure 
of the tuning windows: a value $x$ belongs to $B(\omega) \cap [\omega, \alpha_0]$ if and only if its continued fraction
is an infinite concatenation of the strings $S_0$, $S_1$.

\begin{proposition} \label{concat}
Let $r \in \QQ_E$, and let $W_r = [\omega, \alpha_0)$. Then 
$$\BB(\omega) \cap [\omega, \alpha_0] = K(\Sigma)$$
where $K(\Sigma)$ is the regular Cantor set on the alphabet $\Sigma = \{ S_0, S_1 \}$.
\end{proposition}

For instance, if $r = \frac{1}{2}$, then $W_{\frac{1}{2}} = [g^2, g)$, and $\mathcal{B}(g^2) \cap[g^2, g]$ is 
the set of numbers whose continued fraction expansion is an infinite concatenation of the 
strings $S_0 = (1,1)$ and $S_1 = (2)$.

\subsection{Tuning operators}

For each $r\in \QQ_E$ we can define the {\em tuning map} $\tau_r:[0,1]\to [0,r]$ as $\tau_r(0)=\omega$ and
\begin{equation}\label{eq:tuning}
\tau_r([0; a_1, a_2, \dots]) = [0; S_1 S_0^{a_1-1}S_1S_0^{a_2-1}\dots] 
\end{equation}
Note that this map is well defined even on rational values (where the continued fraction representation is not unique); 
for instance, $\tau_{1/3}([0; 3, 1]) = [0; 3, 2, 1, 2, 1, 3]=[0; 3, 2, 1, 2, 1, 2, 1]= \tau_{1/3}([0; 4])$.

It will be sometimes useful to
consider the action that $\tau_r$ induces on finite strings of
positive integers: with a slight abuse of notation we shall denote
this action by the same symbol $\tau_r$.

\begin{lemma} \label{inj}
For each $r \in \mathbb{Q}_E$, the map $\tau_r$ is strictly increasing (hence injective).
Moreover, $\tau_r$ is continuous at all irrational points, and discontinuous at every positive rational number.
\end{lemma}

The first key feature of tuning operators is that 
they map the bifurcation set into a small copy of itself:

\begin{proposition}\label{qrenorm}
Let $r\in \QQ_E$. Then  
\begin{enumerate}
\item[(i)] $\tau_r(\EE) = \EE \cap W_r$, and $\tau_r$ is a homeomorphism of $\EE$ onto $\EE \cap W_r$;
\item[(ii)] $\tau_r(\QQ_E)=\QQ_E \cap W_r\setminus \{r\}$. 
\end{enumerate}
\end{proposition}
Let us moreover notice that tuning windows are nested:

\begin{lemma} \label{nested}
Let $r, s \in \mathbb{Q}_E$. Then the following are equivalent:
\begin{enumerate}
 \item[(i)] 
$\overline{W_r} \cap \overline{W_s} \neq \emptyset \textup{ with }r < s$; 
\item[(ii)] $r = \tau_s(p) \textup{ for some }p 
\in \mathbb{Q}_E$;
\item[(iii)]$\overline{W_r} \subseteq W_s$.
\end{enumerate}
\end{lemma}


\subsection{Proofs}

\begin{proof}[Proof of lemma \ref{inj}]
Let us first prove that $\tau_r$ preserves the order between irrational numbers. Pick $\alpha, \beta \in (0,1) \setminus \mathbb{Q}$, 
 $\alpha\neq \beta$. Then
$$\alpha:=[0;P,a,a_2,a_3,...], \ \ \ \ \beta:=[0;P,b,b_2,b_3,...] $$
where $P$ is a finite string of positive integers (common prefix), and
we may assume also that $a<b$. Then
$$\tau_r(\alpha):=[0;\tau_r(P),S_1,S_0^{a-1},S_1,...],
\ \ \ \ \tau_r(\beta):=[0;\tau_r(P),S_1,S_0^{b-1},S_1,...].$$ 
Since $|S_0^{a-1}|$ is even and
$S_1<<S_0$, we get $S_0^{a-1}S_1<<S_0^{b-1}S_1$, whence
$S_1S_0^{a-1}S_1>>S_1S_0^{b-1}S_1$. Therefore, since $|P|\equiv|\tau_r(P)|
\mod 2$, we get that either $|P|$ is even, $\alpha>\beta$ and
$\tau_r(\alpha)>\tau_r(\beta)$, or 
$|P|$ is odd , $\alpha<\beta$ and
$\tau_r(\alpha)<\tau_r(\beta)$, so we are done. 
The continuity of $\tau_r$ at irrational points follows from the fact that if $\beta \in (0,1)\setminus \QQ$ and $x$ is close to $\beta$ then the continued fraction expansions of $x$ and $\beta$ have a long common prefix, 
and, by definition of $\tau_r$, 
then their images will also have a long prefix in common, and will therefore be close to each other.
Finally, let us check that the function is increasing at each rational number $c > 0$. This follows from the property:
\begin{equation} \label{ratlgap}
\sup_{\stackrel{\alpha \in \mathbb{R}\setminus \mathbb{Q}}{\alpha < c}} \tau_r(\alpha) < \tau_r(c) < \inf_{\stackrel{\alpha \in\mathbb{R} \setminus \mathbb{Q}}{\alpha > c}}\tau_r(\alpha)
\end{equation}
Let us prove the left-hand side inequality of \eqref{ratlgap} (the right-hand side one has essentially the same proof).
Suppose $c = [0; S]$, with $|S| \equiv 1 \mod 2$. Then every irrational $\alpha < c$ has an expansion of the form $\alpha = [0; S, A]$ with $A$ an 
infinite string. Hence $\tau_r(\alpha) = [0; \tau_r(S), \tau_r(A)]$, and it is not hard to check that 
$\sup \tau_r(\alpha) = [0; \tau_r(S), S_1, \overline{S_0}] < [0; \tau_r(S)] = \tau_r(c)$.
Discontinuity at positive rational points also follows from \eqref{ratlgap}.
\end{proof}

To prove Propositions   \ref{concat}
and \ref{qrenorm} we first need some lemmata.

\begin{lemma}\label{prosuf}
Let $r=[0;S_0]=[0;S_1]\in \QQ_E$ and  $y$ be an irrational number with c.f. expansion $y = [0;  B, S_*, \dots]$, 
where $B$ is a proper suffix of either $S_0$ or $S_1$, and $S_*$ equal to either $S_0$ or $S_1$. Then $y> [0;S_1]$.
\end{lemma}
\begin{proof}
If $B=(1)$ then there is hardly anything to prove (by Prop. \ref{stringlemma}, the first digit of $S_1$ is strictly greater than $1$).
If not, then one of the following is true:
\begin{enumerate}
\item $S_0=AB$ and $A$ is a prefix of $S_1$ as well; 
\item $S_1=AB$ and $A$ is a prefix of $S_0$ as well. 
\end{enumerate}
By Prop. \ref{stringlemma}, in the first case we get that $BA\geq AB=S_0>>S_1$, while in the latter $BA>>AB=S_1$; so in both cases
$BA>>S_1$ and the claim follows.
\end{proof}

\begin{lemma} \label{induced_order}
Let $r \in \mathbb{Q}_E$, and $x, y \in [0, 1] \setminus \mathbb{Q}$. Then
$$G^k(x) \geq y \ \ \forall k \geq 0$$
if and only if 
$$G^k(\tau_r(x)) \geq \tau_r(y) \ \ \forall k \geq 0$$
\end{lemma}

\begin{proof}
Since $\tau_r$ is increasing, $G^k(x) \geq y$ if and only if
$\tau_r(G^k(x)) \geq \tau_r(y)$ if and only if $G^{N_k}(\tau_r(x))
\geq \tau_r(y)$ for $N_k = |S_0|(a_1 + \dots + a_k) +
(|S_1|-|S_0|)k$. 

On the other hand, if $h$ is not of the form $N_k$,
$G^h(\tau_r(x)) = [0; B, S_*, \dots]$ with $B$ a proper suffix of either $S_0$
or $S_1$, and $S_*$ equal to either $S_0$ or $S_1$.  By
Lemma \ref{prosuf} it follows immediately that 
$$G^h(\tau_r(x)) > [0; S_1] \geq \tau_r(y)$$
\end{proof}

\begin{proof}[Proof of Proposition \ref{concat}]
  Let us first
prove that, if $x\in \BB(\omega) \cap [\omega, \alpha_0] $ then
$x=S\cdot y$ with $y\in \BB(\omega) \cap [\omega, \alpha_0]$ and $S\in
\{S_0,S_1\}$; then the inclusion 
$$\BB(\omega) \cap [\omega, \alpha_0] \subset K(\Sigma)$$
will follow by induction. If $x\in \BB(\omega) \cap [\omega, \alpha_0]$ then
the following alternative holds
\begin{enumerate}
\item[($x>r$)] $x=S_0\cdot y$ and $ S_0\cdot y=x<\alpha_0=S_0\cdot \alpha_0$, therefore $ y \leq \alpha_0$;
\item[($x<r$)] $x=S_1\cdot y$ and $ S_1\cdot y=x>\omega=S_1\cdot \alpha_0$, therefore $ y \leq \alpha_0$;
\end{enumerate}
Note that, since the map $y\mapsto S\cdot y$ preserves or reverses the
order depending on the parity of $|S|$, in both cases we get to the
same conclusion. Moreover, since $\BB(\omega)$ is forward-invariant 
with respect to the Gauss map and $x\in \mathcal{B}(\omega)$, then $y=G^k(x) \in \BB(\omega)$ as well, hence $y\in \BB(\omega) \cap [\omega,
  \alpha_0]$.

To prove the other inclusion, let us first remark that every $x\in
K(\Sigma)$ satisfies $\omega\leq x\leq \alpha_0$. Now, let $k\in \NN$;
either $G^k(x)\in K(\Sigma)$, and hence $G^k(x) \geq \omega$, or
$G^k(x)=[0;B,S_*,...]$ satisfies the hypotheses of Lemma \ref{prosuf},
and hence we get that $y> [0;S_1]>\omega$. Since $G^k(x) \geq\omega$ holds for
any $k$, then $x\in \BB(\omega)$.
\end{proof}

\medskip
\begin{proof}[Proof of Proposition \ref{qrenorm}]
(i) Recall the notation $W_r = [\omega, \alpha_0)$, and let $v \in \EE \cap W_r$. By Remark \ref{inclusion}, $\EE \cap W_r \subseteq \BB(\omega) \cap [\omega, \alpha_0)$, 
hence, by Proposition \ref{concat}, $v\in K(\Sigma)$.
Moreover, $v < r$ because $\EE \cap [r, \alpha_0) = \emptyset$.
As a consequence, the c.f. expansion of $v$ is an infinite concatenation of strings in the alphabet $\{S_0, S_1\}$ starting 
with $S_1$. Now, if the expansion of $v$ terminates with $\overline{S_0}$, then 
$G^k(v) = \omega$ for some $k$, hence $v$ must coincide 
with $\omega = [0; S_1 \overline{S_0}]$, so  $v=\tau_r(0)$ and we are done. Otherwise, 
there exists some $x \in [0, 1)$ such that $v=\tau_r(x)$: then  
by Lemma \ref{induced_order} we get that
$$G^k(v) \geq v  \ \ \forall k \geq 0  \Rightarrow 
G^k(x) \geq x  \ \ \forall k \geq 0
$$ 
which means $x$ belongs to $\EE$.

Viceversa, let us pick $x:= \tau_r(v)$ with $v \in \EE$. By definition of $\tau_r$, 
$x \in W_r$. Moreover, since $v$ belongs to $\EE$, $G^n(v) \geq v$ for any $n$, 
hence by Lemma \ref{induced_order} also $\tau_r(v)$ belongs to $\EE$.
The fact that $\tau_r$ is a homeomorphism follows from bijectivity and compactness.


(ii) Let $p\in \QQ_E$ and $I_p = (\alpha_1, \alpha_0)$ the maximal quadratic interval generated by $p$; 
by point (i) above also the values $\beta_i:=\tau_r(\alpha_i)$, $(i=0,1)$ belong to $\EE \cap W_r$. Since $\tau_r$
is strictly increasing, no other point of $\EE$ lies between $\beta_1$ and $\beta_0$, hence $(\beta_1, \beta_0)= I_s$ for some $s \in
\QQ_E\cap[\omega, r)$.  Since $\tau_r(p)$ is a convergent to both $\tau_r(\alpha_0)$ and $\tau_r(\alpha_1)$, then $\tau_r(p) = s$.

To prove the converse, pick $s\in \QQ_E\cap [\omega,r)$ and denote $I_s=(\beta_1, \beta_0)$. Again  by point (i),
$\beta_i:=\tau_r(\alpha_i)$ for some $\alpha_0, \alpha_1 \in \EE$, and $(\alpha_1, \alpha_0)$ is a component of
the complement of $\EE$, hence there exists $p \in \mathbb{Q}_E$ such that $I_p = (\alpha_1, \alpha_0)$.
As a consequence, $s=\tau_r(p)$.  
\end{proof}

\begin{proof}[Proof of lemma \ref{nested}]
Let us denote $W_s = [\omega(s), \alpha_0(s))$, $W_r = [\omega(r), \alpha_0(r))$, $W_p = [\omega(p), \alpha_0(p))$.
Suppose (i): 
then, since the closures of $W_r$ and $W_s$ are not disjoint, $\omega(s) \leq \alpha_0(r)$. Moreover, $\omega(s) \in\mathcal{E}$ and 
$\mathcal{E} \cap (r, \alpha_0(r)] = \{\alpha_0(r)\}$, hence $\omega(s) \leq r$ because $\omega(s)$ cannot 
coincide with $\alpha_0(r)$, not having a purely periodic c.f. expansion. Hence $r \in W_s$ and, by Proposition \ref{qrenorm}, 
there exists $p \in \mathbb{Q}_E$ such that $r = \tau_s(p)$. 

Suppose now (ii). Then, since $r = \tau_s(p)$, also $\alpha_0(r) = \tau_s(\alpha_0(p)) \leq s < \alpha_0(s)$, and 
$\omega(r) = \tau_s(\omega(p)) \in W_s$, which implies (iii).

(iii) $\Rightarrow$ (i) is clear.
\end{proof}

\section{Tuning and monotonicity of entropy: proof of Theorem \ref{main}} \label{sec:mono}

\begin{definition}
Let $A=(a_1,...,a_n)$ be a string of positive integers. Then its \emph{matching index} $\llbracket A \rrbracket$ is the alternating sum
of its digits: 
\begin{equation}\label{eq:altersum}
\llbracket A \rrbracket := \sum_{j=1}^{n} (-1)^{j+1} a_j
\end{equation}
Moreover, if $r = [0; S_0]$ is a rational number between $0$ and $1$ and $S_0$ is its continued fraction expansion of even length, 
we define the \emph{matching index} of $r$ to be
$$\llbracket r \rrbracket := \llbracket S_0 \rrbracket$$
\end{definition}

The reason for this terminology is the following. Suppose $r \in \mathbb{Q}_E$ is the pseudocenter of the maximal quadratic interval $I_r$:
then by Theorem \ref{matching}, a matching condition \eqref{eq:matching} holds, and by formula \eqref{eq:index}
\begin{equation} \label{eq:MNindex}
\llbracket r \rrbracket = \sum_{j = 1}^n (-1)^{j+1} a_j = M - N
 \end{equation}
where $r = [0; S_0]$ and $S_0 = (a_1, \dots, a_n)$.
This means, by Proposition \eqref{Nakada}, that the entropy function $h(\alpha)$ is increasing on $I_r$ iff $\llbracket r \rrbracket > 0$, 
decreasing on $I_r$ iff $\llbracket r \rrbracket < 0$, and constant on $I_r$ iff $\llbracket r \rrbracket = 0$. 

\begin{lemma}\label{ss} Let $r, p \in \mathbb{Q}_E$. Then 
\begin{equation}\label{eq:product} 
\llbracket \tau_r(p)\rrbracket = - \llbracket r \rrbracket
\llbracket p \rrbracket.
\end{equation}

\end{lemma}

\begin{proof}
The double bracket notation behaves well under concatenation, namely:
$$
\llbracket AB \rrbracket := \left\{ 
\begin{array}{ll}
\llbracket A \rrbracket+\llbracket B \rrbracket & \textup{if }|A| \mbox{ even }\\
\llbracket A \rrbracket-\llbracket B \rrbracket & \textup{if }|A| \mbox{ odd }
\end{array}
\right.
$$

Let $p = [0; a_1,...,a_n]$ and $r=[0;S_0]$ be the continued fraction expansions of even length of $p, r\in \QQ_E$; using the definition of $\tau_r$ we get  
$$
\llbracket \tau_r(p)\rrbracket = \sum_{j=1}^{n} (-1)^{j+1}
 \left(\llbracket S_1 \rrbracket-(a_j -1) \llbracket S_0 \rrbracket\right)
$$
and, since $n=|A|$ is even, the right-hand side becomes $ \llbracket S_0 \rrbracket \sum_{j=1}^{n}(-1)^j a_j$, whence
 the thesis.
\end{proof}

\begin{definition}
A quadratic interval $I_r$ is called \emph{neutral} if $\llbracket r \rrbracket = 0$. Similarly, 
a tuning window $W_r$ is called \emph{neutral} if $\llbracket r \rrbracket = 0$. 
\end{definition}
As an example, the rational $r = \frac{1}{2} = [0; 2] = [0; 1, 1]$
generates the neutral tuning window $W_{1/2} = [g^2, g)$. 

\medskip

{ \bf Proof of Theorem \ref{main}.}  
Let $I_r$ be a maximal quadratic interval over which the entropy is increasing. Then, by Theorem \ref{matching} and Proposition \ref{Nakada}, 
for $\alpha \in I_r$, a matching condition \eqref{eq:matching} holds, with $M - N > 0$. This implies by \eqref{eq:MNindex} that
$\llbracket r \rrbracket > 0$.
Let now $I_p$ be another maximal quadratic interval. By Proposition \ref{qrenorm} (ii), $I_{\tau_r(p)}$ is 
also a maximal quadratic interval, and by Lemma \ref{ss}
$$\llbracket \tau_r(p) \rrbracket = - \llbracket r \rrbracket \llbracket p \rrbracket$$
Since $\llbracket r \rrbracket > 0$, then $\llbracket \tau_r(p) \rrbracket$ and $\llbracket p \rrbracket$ have opposite sign. 
In terms of the monotonicity of entropy, this means the following:

\begin{enumerate}
 \item if the entropy is increasing on $I_p$, then by \eqref{eq:MNindex} $\llbracket p \rrbracket > 0$, hence $\llbracket \tau_r(p) \rrbracket < 0$, which 
implies (again by \eqref{eq:MNindex}) that the entropy is decreasing on $I_{\tau_r}(p)$;
 \item if the entropy is decreasing on $I_p$, then $\llbracket p \rrbracket < 0$, hence $\llbracket \tau_r(p) \rrbracket > 0$ and 
the entropy is increasing on $I_{\tau_r}(p)$;
 \item if the entropy is constant on $I_p$, then $\llbracket p \rrbracket = 0$, hence $\llbracket \tau_r(p) \rrbracket = 0$ and 
the entropy is constant on $I_{\tau_r}(p)$.
\end{enumerate}
If, instead, the entropy is decreasing on $I_r$, then $\llbracket r \rrbracket > 0$, hence $\llbracket \tau_r(p) \rrbracket$ and 
$\llbracket p \rrbracket$ have the same sign, which similarly to the previous case implies that the 
monotonicity of entropy on $I_p$ and $I_{\tau_r(p)}$ is the same.\qed

\begin{remark} \label{neutral}
The same argument as in the proof of Theorem \ref{main} shows that, if $r \in \mathbb{Q}_E$ with $\llbracket r \rrbracket = 0$, 
then the entropy on $I_{\tau_r(p)}$ is constant for each $p \in \mathbb{Q}_E$ (no matter what the monotonicity is on $I_p$). 
\end{remark}

\section{Plateaux: proof of Theorem \ref{plateaux}}\label{hoelder}

The goal of this section is to prove Theorem \ref{last}, 
which characterizes the plateaux of the entropy and has as a consequence 
Theorem \ref{plateaux} in the introduction. Meanwhile, we introduce the sets 
of \emph{untuned parameters} (subsection \ref{ss:ut}) and \emph{dominant parameters} (subsection \ref{ss:dom})
which we will use in the proof of the Theorem (subsection \ref{ss:proof}).

\subsection{The importance of being H\"older}
The first step in the proof of Theorem \ref{plateaux} is proving that the entropy function 
$h(\alpha)$ is indeed constant on neutral tuning windows:

\begin{proposition}\label{neutraltuning}
Let $r \in \mathbb{Q}_E$ generate a neutral maximal interval,
i.e. $\llbracket r \rrbracket = 0$.  Then the entropy function
$h(\alpha)$ is constant on $\overline{W_r}$.
\end{proposition} 

By Remark \ref{neutral}, we already know that 
the entropy is locally constant on all connected components of $W_r\setminus \EE$, 
which has full measure in $W_r$.
However, since $W_r\cap \EE$ has, in general, positive Hausdorff
dimension, in order to prove that the entropy is actually constant on the whole $W_r$ 
one needs to exclude a devil staircase behaviour. We shall exploit the following criterion:

\begin{lemma} \label{extend}
Let $f : I \rightarrow \mathbb{R}$ be a H\"older-continuous function of exponent $\eta \in (0,1)$, and assume that
there exists a closed set $C \subseteq I$  such that $f$ is locally constant at all $x\notin C$. Suppose moreover
$\textup{H.dim }C < \eta$. Then $f$ is constant on $I$.
\end{lemma}

\begin{proof}
Suppose $f$ is not constant: then by continuity $f(I)$ is an interval with non-empty interior, hence $\textup{H.dim } f(I)=1$. 
On the other hand, we know $f$ is constant on the connected components of $I\setminus C$, so we get $f(I)=f(C)$, whence
$$\textup{H.dim }f(C)=\textup{H.dim }f(I)=1.$$
But, since $f$ is $\eta$-H\"older continuous, we also get (e.g.  by \cite{F}, prop. 2.3)
$$\textup{H.dim }f(C)\leq \frac{ \textup{H.dim }C }{\eta}$$
and thus  $\eta \leq \textup{H.dim }C$, contradiction. 
\end{proof}

Let us know check the hypotheses of Lemma \ref{extend} are met in our case; the first one is given by the following 

\begin{theorem}[\cite{T}] \label{hexp}
For all fixed $0< \eta < 1/2$, the function $\alpha \mapsto h(\alpha)$ is locally H\"older-continuous of 
exponent $\eta$ on $(0,1]$.
\end{theorem}

We are now left with checking that the Hausdorff dimension of $\mathcal{E} \cap W_r$ is small enough:

\begin{lemma}\label{TW}  
For all $r \in \QQ_E$, 
$$\textup{H.dim } \EE \cap W_r \leq \frac{\log 2}{\log 5}<1/2.$$
\end{lemma}

\begin{proof}
Let $r\in \QQ_E$, $r=[0;S_0]=[0;S_1]$ and $\overline{W_r}=[\omega, \alpha]$. By Remark \ref{inclusion} and Proposition \ref{concat}, 
$$\EE \cap W_r \subset \BB(\omega) \cap [\omega,\alpha]=K(\Sigma), \ \ \ \ \mbox{ with } \ \ \ \Sigma = \{ S_0, S_1\}.$$
Note we also have 
$K(\Sigma) = K(\Sigma_2)$ with $\Sigma_2 = \{S_0S_0,S_1S_0,S_1S_0,S_1S_1\}$ and, by virtue of \eqref{contraction} we have the estimate 
$$ |f'_{S_iS_j}(x)|\leq \frac{1}{q(S_iS_j)^2}, \ \ \ i,j \in \{0,1\}.$$
On the other hand, setting $Z_0=(1,1)$ and $Z_1=(2)$ we can easily check that 
$$ q(S_iS_j)\geq  q(Z_iZ_j)=5 \ \ \ \forall i,j \in \{0,1\};$$
whence $|f'_{S_iS_j}(x)|\leq \frac{1}{25}$ and, by formula \eqref{dimbounds}, we get our claim.  
\end{proof}
Proposition \ref{neutraltuning} now follows from Lemma \ref{extend}, Theorem \ref{hexp} and Lemma \ref{TW}.


\subsection{Untuned parameters} \label{ss:ut}

The set of \emph{untuned parameters} is the complement of all tuning windows:
$$UT := [0, g] \setminus \bigcup_{r \in \mathbb{Q} \cap (0, 1)} W_r$$

Note that, since $I_r \subseteq W_r$, $UT \subseteq \mathcal{E}$. 
Moreover, we say that a rational $a\in \QQ_E$ is {\em untuned} if it cannot be written as 
$a=\tau_r(a_0)$ for some $r,a_0 \in \QQ_E$.
We shall denote by $\QUT$ the set of all $a\in \QQ_E$ which are untuned.
Let us start out by seeing that each pseudocenter of a maximal quadratic interval admits an ``untuned factorization'':

\begin{lemma} \label{factorization}
Each $r \in \mathbb{Q}_E$ can be written as:
\begin{equation}\label{eq:utfactor}
r = \tau_{r_m} \circ \dots \circ \tau_{r_1}(r_0),  \ \ \ \ \textup{with } r_i \in \QUT \ \forall i\in \{0,1,...,m\}.
\end{equation}
Note that $m$ can very well be zero (when $r$ is already untuned).
\end{lemma}

\begin{proof}
A straightforward check shows that the tuning operator has the following
associativity property:
\begin{equation}\label{eq:associativity}
\tau_{\tau_p(r)}(x)=\tau_p \circ \tau_r (x)  \qquad \forall p,r \in \QQ_E,  \ x\in (0,1)
\end{equation}
For $s=[0;a_1,...,a_m]\in \QQ_E$ we shall set $\|s\|_1:=\sum_{1}^{m}a_i$; this definition does not depend on the representation 
of $s$, moreover
$$\|\tau_p(s)\|_1=\|p\|_1 \|s\|_1  \qquad \forall p,s \in \QQ_E$$
The proof of \eqref{eq:utfactor} follows then easily by induction on $N=\|r\|_1$, using the fact that
$\max (\|p\|_1,\|s\|_1)\leq \|\tau_p(s)\|_1 /2$. 
\end{proof}

As a consequence of the following proposition, the connected components of the complement of $UT$ 
are precisely the tuning windows generated by the elements of $\mathbb{Q}_{UT}$:
 
\begin{proposition} \label{cantorut} 
\begin{itemize} The set $\overline{UT}$ is a Cantor set: indeed, 
 \item[(i)]
$$UT = [0, g] \setminus \bigcup_{r \in \mathbb{Q}_{UT}} W_r;$$
\item[(ii)] 
if $r, s \in \mathbb{Q}_{UT}$ with $r \neq s$, then $\overline{W_r}$ and $\overline{W_s}$ are disjoint;
\item[(iii)]
if $x \in \overline{UT} \setminus UT$, then there exists $r \in \mathbb{Q}_{UT}$ such that $x = \tau_r(0)$.
\end{itemize}
\end{proposition}

\begin{proof}
(i). It is enough to prove that every tuning window $W_r$ is contained
  in a tuning window $W_s$, with $s \in \mathbb{Q}_{UT}$. Indeed, let
  $r \in \mathbb{Q}_E$; either $r \in \mathbb{Q}_{UT}$ or, by Lemma
  \ref{factorization}, there exists $p \in \mathbb{Q}_E$ and $s \in
  \mathbb{Q}_{UT}$ such that $r = \tau_s(p)$, hence $W_r \subseteq
  W_s$.
 
(ii). By Lemma \ref{nested}, if the closures of $W_r$ and $W_s$ are
  not disjoint, then $r = \tau_s(p)$, which contradicts the fact $r
  \in \mathbb{Q}_{UT}$.

(iii). By (i) and (ii), $\overline{UT}$ is a Cantor set, and each element $x$ which 
belongs to $\overline{UT} \setminus UT$ 
is the left endpoint of some tuning window $W_r$ with $r \in \QUT$, which 
is equivalent to say $x = \tau_r(0)$.
\end{proof}

\begin{lemma} \label{HDUT}
The Hausdorff dimension of $UT$ is full:
$$\textup{H.dim }UT = 1$$
\end{lemma}
\begin{proof}
By the properties of Hausdorff dimension,  
$$\textup{H.dim } \mathcal{E} = \max \{ \textup{H.dim }UT, \sup_{r\in
  \mathbb{Q}_{UT}} \textup{H.dim }\mathcal{E}\cap W_r \}$$ Now, by
\cite{CT}, $\textup{H.dim }\mathcal{E}= 1$, and, by Lemma \ref{TW},
$\textup{H.dim }\mathcal{E} \cap W_r < \frac{1}{2}$, hence the claim.
\end{proof}

\subsection{Dominant parameters} \label{ss:dom}

\begin{definition}
A finite string $S$ of positive integers is \emph{dominant} if it has even length and 
$$S = AB << B$$
for any splitting $S = AB$ of $S$ into two non-empty strings $A$, $B$.
\end{definition}

That is to say, dominant strings are smaller than all their proper suffixes. A related definition is 
the following:

\begin{definition}
A quadratic irrational $\alpha \in [0, 1]$ is a \emph{dominant parameter} if its c.f. expansion is of the form 
$\alpha = [0; \overline{S}]$ with $S$ a dominant string.
\end{definition}

For instance, $(2, 1, 1, 1)$ is dominant, while $(2, 1, 1, 2)$ is not (it is not true that 
$(2, 1, 1, 2) << (2)$). In general, all strings whose first digit is strictly greater than the others are dominant, but there are even 
more dominant strings (for instance $(3, 1, 3, 2)$ is dominant).

\begin{remark} \label{domextreme}
By Proposition \ref{stringlemma}, if $S$ is dominant then 
$[0; S]\in \QQ_E$.
\end{remark}

A very useful feature of dominant strings is that they can be easily used to produce other dominant strings:
\begin{lemma} \label{newdom}
Let $S_0$ be a dominant string, and $B$ a proper suffix of $S_0$ of even length. Then, for any $m \geq 1$, $S_0^m B$ is a dominant string.
\end{lemma}

\begin{proof}
Let $Y$ be a proper suffix of $S_0^mB$. There are three possible cases:
\begin{enumerate}
 \item $Y$ is a suffix of $B$, hence a proper suffix of $S_0$. Hence, since $S_0$ is dominant, $S_0  >> Y$ and $S_0^m B >> Y$.
\item $Y$ is of the form $S_0^kB$, with $1 \leq k < m$. Then by dominance $S_0 >> B$, which implies $S_0^{m-k}B >> B$, hence $S_0^m B >> S_0^k B$.
\item $Y$ is of the form $CS_0^k B$, with $0 \leq k < m$ and $C$ a proper suffix of $S_0$. Then again the claim follows by the fact 
that $S_0$ is dominant, hence $S_0 >> C$.
\end{enumerate}
\end{proof}

\begin{lemma} \label{notwoconsec}
A dominant string $S_0$  cannot begin with two equal digits.
\end{lemma}

\begin{proof}
By definition of dominance, $S_0$ cannot consist of just $k \geq 2$ equal digits. Suppose instead 
it has the form $S_0 = (a)^k B$ with $k \geq 2$ and $B$ non empty and which does not begin with $a$. Then 
by dominance $(a)^k B < < B$, hence $a << B$ since $B$ does not begin with $a$. 
However, this implies $aB << aa$ and hence $aB << (a)^kB = S_0$, which contradicts the definition 
of dominance because $aB$ is a proper suffix of $S_0$.
\end{proof}

The reason why dominant parameters turn out to be so useful is that they can approximate untuned parameters:

\begin{proposition}[\cite{CT2}, Proposition 6] \label{density}
The set of dominant parameters is dense in $UT \setminus \{ g\}$. More precisely, every parameter in $UT \setminus \{g\}$
is accumulated from the right by a dominant parameter.
\end{proposition}

\begin{proposition} \label{rightaprrox}
Every element $\beta \in \overline{UT} \setminus \{ g\}$ is
accumulated by non-neutral maximal quadratic intervals. 
\end{proposition}

\begin{proof}
We shall prove that either $\beta \in UT \setminus \{ g\}$, and $\beta$ is
accumulated from the right by non-neutral maximal quadratic intervals,
or $\beta=\tau_s(0)$ for some $s\in \QUT$, and $\beta$ is accumulated
from the left by non-neutral maximal quadratic intervals.

If $\beta\in UT$ then, by Proposition \ref{density}, $\beta$ is the limit point from the right of a sequence $\alpha_n = [0; \overline{A_n}]$ 
with $A_n$ dominant. If $\llbracket A_n \rrbracket \neq 0$ for infinitely many $n$, the claim is proven. 
Otherwise, it is sufficient to prove that every dominant parameter $\alpha_n$ such that $\llbracket A_n \rrbracket = 0$
is accumulated from the right by non-neutral maximal intervals. Let $S_0$ be a dominant string, with $\llbracket S_0 \rrbracket = 0$, 
and let $\alpha := [0; \overline{S_0}].$
First of all, the length of $S_0$ is bigger than $2$: indeed, if $S_0$ had length $2$, then 
condition $\llbracket S_0 \rrbracket = 0$ would force it to be of the form $S_0 = (a, a)$ for some $a$, which contradicts
the definition of dominant. Hence, we can write $S_0 = AB$ with $A$ of length $2$ and $B$ of positive, even length.
Then, by Lemma \ref{newdom}, $S_0^m B$ is also dominant, hence $p_m := [0; S_0^m B] \in \mathbb{Q}_E$ by Remark \ref{domextreme}. Moreover,
$\alpha < p_m$ since $S_0 << B$. 
Furthermore, $S_0$ cannot begin with two equal digits (Lemma \ref{notwoconsec}), hence $\llbracket A \rrbracket \neq 0$ 
and $\llbracket S_0^m B \rrbracket = \llbracket B \rrbracket = \llbracket S_0 \rrbracket - \llbracket A \rrbracket \neq 0$.
Thus the sequence $I_{p_m}$ is a sequence of non-neutral maximal quadratic intervals
which tends to $\beta$ from the right, and the claim is proven.  

If $\beta \in \overline{UT} \setminus UT$, then by Proposition \ref{cantorut} (iii) there exists $s\in \QUT$ such that 
$\beta=\tau_s(0)$. Since $\overline{UT}$ is a Cantor set and $\beta$ lies on its boundary, $\beta$ is the limit point (from the left)
of a sequence of points of $UT$, hence the claim follows by the above discussion.
\end{proof}



\subsection{Characterization of plateaux} \label{ss:proof}

\begin{definition}
A parameter $x \in \mathcal{E}$ is \emph{finitely renormalizable} if it belongs to finitely many tuning windows. This is equivalent to 
say that $x = \tau_r(y)$, with $y \in UT$. A parameter $x \in \mathcal{E}$ is \emph{infinitely renormalizable} if it lies in infinitely many tuning windows $W_r$, with $r \in \mathbb{Q}_E$.
Untuned parameters are also referred to as {\em non renormalizable}.
\end{definition}

We are finally ready to prove Theorem \ref{plateaux} stated in the introduction, and indeed the following stronger version:

\begin{theorem} \label{last}
An open interval $U \subseteq [0, 1]$ of the parameter space of $\alpha$-continued fraction transformations is a plateau for the entropy
function $h(\alpha)$ if and only if it is the interior of a neutral tuning window $U = \overset{\circ}{W_r}$, with $r$ of either one of the following types: 

$$\begin{array}{llr}
(NR) &  r \in \mathbb{Q}_{UT}, \llbracket r \rrbracket = 0 & \textup{(non-renormalizable case)} \\
(FR) &  r=\tau_{r_1}(r_0) \ with 
\left\{
\begin{array}{ll} 
r_0\in \QUT, & \llbracket r_0 \rrbracket =0\\
r_1\in \QQ_E, & \llbracket r_1 \rrbracket \neq 0
\end{array} \right. &  \textup{(finitely renormalizable case)} 
\end{array}$$
\end{theorem}

\begin{proof}
Let us pick $r$ which satisfies (NR), and let $W_r = [\omega, \alpha_0)$ be its tuning window.
By Proposition \ref{neutraltuning}, since $\llbracket r \rrbracket =0$, the entropy is constant on $\overline{W_r}$.
Let us prove that it is not constant on any larger interval.
Since $r \in \mathbb{Q}_{UT}$, by Proposition \ref{cantorut}, $\alpha_0$ belongs to $UT$. 
If $\alpha_0 = g$, then by the explicit formula \eqref{eq:N81}
the entropy is decreasing to the right of $\alpha_0$.
Otherwise, by Proposition \ref{rightaprrox}, $\alpha_0$ is accumulated from the right by non-neutral maximal quadratic intervals, 
hence entropy is not constant to the right of $\alpha_0$.
Moreover, by Proposition \ref{cantorut}, $\omega$ belongs to the boundary of $UT$, hence, 
by Proposition \ref{rightaprrox}, it is accumulated from the left by non-neutral intervals. 
This means that the interior of $W_r$ is a maximal open interval of constance for the entropy $h(\alpha)$, i.e. a plateau.

Now, suppose that $r$  satisfies condition (FR), with $r = \tau_{r_1}(r_0)$. By the (NR) case, the interior of $W_{r_0}$ is a plateau,
and $W_{r_0}$ is accumulated from both sides by non-neutral
intervals. Since $\tau_{r_1}$ maps non-neutral intervals to
non-neutral intervals and is continuous on $\EE$, then $W_r$ is
accumulated from both sides by non-neutral intervals, hence its
interior is a plateau.

Suppose now $U$ is a plateau. Since $\mathcal{E}$ has no interior part (e.g. by formula \eqref{Edef}), there is $r \in \mathbb{Q}_E$ such that $I_r$ 
intersects $U$, hence, by Proposition \ref{Nakada}, $\llbracket r \rrbracket = 0$ and actually $I_r \subseteq U$. 
Then, by Lemma \ref{factorization} one has the factorization 
$$r = \tau_{r_n} \circ \dots \circ \tau_{r_1} (r_0)$$
with each $r_i \in \mathbb{Q}_{UT}$ untuned (recall $n$ can possibly be zero, in which case $r = r_0$). 
Since the matching index is multiplicative (eq. \eqref{eq:product}), there exists at least one $r_i$
with zero meatching index: let $j \in \{0, \dots, n\}$ be the largest index such that $\llbracket r_j \rrbracket = 0$.
If $j = n$, let $s := r_n$: by the first part of the proof, the interior of $W_s$ is a plateau, 
and it intersects $U$ because they both contain $r$ (by lemma \ref{nested}, $r$ belongs to the interior of $W_s$), 
hence $U = \overset{\circ}{W_s}$, and we are in case (NR). 

If, otherwise, $j < n$, let $s:= \tau_{r_n} \circ \dots \circ \tau_{r_{j+1}}(r_j)$.
By associativity of tuning (eq. \eqref{eq:associativity}) we can write 
$$s = \tau_{s_1}(s_0)$$
with $s_0 := r_j$ and $s_1 := \tau_{r_n} \circ \dots \circ \tau_{r_{j+2}}(r_{j+1})$.
Moreover, by multiplicativity of the matching index (eq. \eqref{eq:product}) $\llbracket s_1 \rrbracket \neq 0$, 
hence $s$ falls into the case (FR) and by the first part of the proof the interior of $W_s$ 
is a plateau. Also, by construction, $r$ belongs to the image of $\tau_s$, hence 
 it belongs to the interior of $W_s$. As a consequence, $U$ and 
$\overset{\circ}{W_s}$  are intersecting plateaux, hence they must coincide.
\end{proof}

\section{Classification of local monotonic behaviour} \label{sec:class}

\begin{lemma} \label{inftypes}
Any non-neutral tuning window $W_r$ contains infinitely many intervals on which the entropy $h(\alpha)$ is constant, 
infinitely many over which it is increasing, and infinitely many on which it is decreasing.
\end{lemma}

\begin{proof}
Let us consider the following sequences of rational numbers 
$$s_n := [0; n, 1]$$ 
$$t_n := [0; n, n]$$
$$u_n := [0; n+1, n, 1, n]$$
It is not hard to check (e.g. using Proposition \ref{stringlemma}) that $s_n, t_n, u_n$ belong to $\mathbb{Q}_E$.
Moreover, by computing the matching indices one finds that, for $n > 2$,  the entropy $h(\alpha)$ is increasing on $I_{s_n}$, constant on $I_{t_n}$ 
and decreasing on $I_{u_n}$. Since $W_r$ is non-neutral, by Theorem \ref{main} $\tau_r$ either induces the same monotonicity or the opposite one, 
hence the sequences $I_{\tau_r(s_n)}, I_{\tau_r(t_n)}$ and $I_{\tau_r(u_n)}$ are sequences of 
maximal quadratic intervals which lie in $W_r$ and display all three types of monotonic behaviour. 
\end{proof}

\noindent {\bf Proof of Theorem \ref{classmon}}.
Let $\alpha \in [0, 1]$ be a parameter. If $\alpha \notin \mathcal{E}$, then $\alpha$ belongs to some maximal quadratic interval $I_r$, hence  
$h(\alpha)$ is monotone on $I_r$ by Proposition \ref{Nakada}, and by formula \eqref{eq:MNindex} the monotonicity type depends on the sign of $\llbracket r \rrbracket$. 

If $\alpha \in \mathcal{E}$, there are the following cases:
\begin{enumerate}
\item $\alpha = g$. Then $\alpha$ is a phase transition as described by formula \eqref{eq:N81};
\item $\alpha \in UT \setminus \{g\}$. Then, by Proposition \ref{rightaprrox}, $\alpha$ is accumulated from the right 
by  non-neutral tuning windows,
and by Lemma \ref{inftypes} each non-neutral tuning window contains 
infinitely many intervals where the entropy 
is constant, increasing or decreasing; the parameter $\alpha$ has therefore mixed monotonic beahaviour.
\item $\alpha$ is finitely renormalizable. Then one can write $\alpha = \tau_r(y)$, with $y \in UT$. There are three subcases:

\begin{itemize}
 \item[(3a)] $\llbracket r \rrbracket \neq 0$, and $y = g$. Since
   $\tau_r$ maps neutral intervals to neutral intervals and
   non-neutral intervals to non-neutral intervals, the phase
   transition at $y = g$ gets mapped to a phase transition at $\alpha$.
\item[(3b)] $\llbracket r \rrbracket \neq 0$, and $y \neq g$. Then, by
  case (2) $y$ is accumulated from the right by intervals
  with all types of monotonicity, hence so is $\alpha$.
\item[(3c)] If $\llbracket r \rrbracket = 0$, then by using the untuned factorization (Lemma \ref{factorization}) 
one can write 
$$\alpha = \tau_{r_m} \circ \dots \circ \tau_{r_0}(y) \qquad r_i \in \QUT$$
Let now $j \in \{ 0, \dots, m\}$ be the largest index such that $\llbracket r_j \rrbracket = 0$.
If $j = m$, then $\alpha$ belongs to the neutral tuning window $W_{r_m}$:
thus, either $\alpha$ belongs to the interior of $W_{r_m}$ (which means by Proposition \ref{neutraltuning} that the entropy is locally constant 
at $\alpha$), or $\alpha$ coincides with the left endpoint of $W_{r_m}$. In the latter case, $\alpha$ belongs to the boundary 
of $UT$, hence by Proposition \ref{rightaprrox} 
and Lemma \ref{inftypes} it has mixed behaviour.
If $j < m$, then by the same reasoning as above $\tau_{r_j} \circ \dots \circ \tau_{r_0}(y)$ 
either lies inside a plateau or has mixed behaviour, and since the operator
$\tau_{r_m} \circ \dots \circ \tau_{r_{j+1}}$ either respects the monotonicity or reverses it, 
also $\alpha$ either lies inside a plateau or has mixed behaviour.
\end{itemize}
\item $\alpha$ is infinitely renormalizable, i.e. $\alpha$ lies in
  infinitely many tuning windows.  If $\alpha$ lies in at least one
  neutral tuning window $W_r = [\omega, \alpha_0)$, then it must lie
    in its interior, because $\omega$ is not infinitely
    renormalizable. This means, by Proposition \ref{neutraltuning},
    that $h$ must be constant on a neighbourhood of $\alpha$.
    Otherwise, $\alpha$ lies inside infinitely many nested non-neutral
    tuning windows $W_{r_n}$.  Since the sequence the denominators of 
    the rational numbers $r_n$ must be unbounded, the size of $W_{r_n}$
    must be arbitrarily small.  By Lemma \ref{inftypes}, in each
    $W_{r_n}$ there are infinitely many intervals with any
    monotonicity type and $\alpha$ displays mixed behaviour.
\end{enumerate}
\qed

Note that, as a consequence of the previous proof, $\alpha$ is a phase
transition if and only if it is of the form $\alpha = \tau_r(g)$, with
$r \in \mathbb{Q}_E$ and $\llbracket r \rrbracket \neq 0$, hence the
set of phase transitions is countable. Moreover, the set of points of
$\mathcal{E}$ which lie in the interior of a neutral tuning window has
Hausdorff dimension less than $1/2$ by Lemma \ref{TW}.

Finally, the set of parameters for which there is mixed behaviour has
zero Lebesgue measure because it is a subset of $\mathcal{E}$. On the
other hand, it has full Hausdorff dimension because such a set
contains $UT \setminus \{g\}$, and by Lemma \ref{HDUT} $UT$ has full
Hausdorff dimension.

\section*{Appendix: Tuning for quadratic polynomials}

The definition of tuning operators given in section \ref{sec:TW} arises 
from the dictionary between $\alpha$-continued fractions and quadratic polynomials 
first discovered in \cite{BCIT}. In this appendix, we will recall a few facts about complex dynamics 
and show how our construction is related to the combinatorial structure of the Mandelbrot set.

\subsection*{Tuning for quadratic polynomials}
Let $f_c(z) := z^2 + c$ be the family of quadratic polynomials, with
$c \in \mathbb{C}$.  Recall the \emph{Mandelbrot set} $\mathcal{M}$ is
the set of parameters $c \in \mathbb{C}$ such that the orbit of the
critical point $0$ is bounded under the action of $f_c$.

The Mandelbrot set has the remarkable property that near every point of its boundary there are infinitely 
many copies of the whole $\mathcal{M}$, called \emph{baby Mandelbrot sets}.
A \emph{hyperbolic component} $W$ of the Mandelbrot set is an open, connected subset of $\mathcal{M}$
such that all $c \in W$, the orbit of the critical point $f^n(0)$ is attracted to a periodic cycle.

Douady and Hubbard \cite{DH} related the presence of baby copies of $\mathcal{M}$ to renormalization 
in the family of quadratic polynomials. 
More precisely, they associated to any hyperbolic component $W$ a \emph{tuning map} $\iota_W : \mathcal{M} \rightarrow \mathcal{M}$ which 
maps the main cardioid of $\mathcal{M}$ to $W$, and such that the image of the whole $\mathcal{M}$ under $\iota_W$ is a baby copy of $\mathcal{M}$.

\subsection*{External rays}
A coordinate system on the boundary of $\mathcal{M}$ is given by external rays. Indeed, the exterior of the Mandelbrot set 
is biholomorphic to the exterior of the unit disk
$$\Phi : \{ z \in \mathbb{C} \ : \ |z| > 1 \} \rightarrow \mathbb{C} \setminus \mathcal{M}$$
The \emph{external ray} at angle $\theta$ is the image of a ray in the complement of the unit disk: 
$$R(\theta) := \Phi(\{ \rho e^{2 \pi i \theta} \ : \ \rho > 1\})$$
The ray at angle $\theta$ is said to \emph{land} at $c \in \partial\mathcal{M}$ if $\lim_{\rho \to 1} \Phi(\rho e^{2\pi i \theta}) = c$.
This way angles in $\mathbb{R} / \mathbb{Z}$ determine parameters on the boundary of the Mandelbrot set, and it turns out that 
the binary expansions of such angles are related to the dynamics of the map $f_c$.

The tuning map can be described in terms of external angles in the following terms. Let $W$ be a hyperbolic component, and 
$\eta_0$, $\eta_1$ the angles of the two external rays which land on the root of $W$. 
Let $\eta_0 = 0.\overline{\Sigma_0}$ and $\eta_1 = 0.\overline{\Sigma_1}$ be the (purely periodic) binary expansions of the two angles 
which land at the root of $W$. 
Let us define the map $\tau_W : \mathbb{R}/\mathbb{Z} \rightarrow \mathbb{R} /\mathbb{Z}$ in the following way:
$$\theta = 0.\theta_1\theta_2\theta_3 \dots \mapsto \tau_W(\theta) = 0.\Sigma_{\theta_1}\Sigma_{\theta_2}\Sigma_{\theta_3}\dots$$
where $\theta = 0.\theta_1 \theta_2 \dots$ is the binary expansion of $\theta$, and its image is given by substituting the binary 
string $\Sigma_0$ to every occurrence of $0$ and $\Sigma_1$ to every occurrence of $1$.

\begin{proposition}[\cite{Do}, Proposition 7]
The map $\tau_W$ has the property that, if the ray of external angle $\theta$ lands at $c \in \partial\mathcal{M}$, 
then the ray at external angle $\tau_W(\theta)$ lands at $\iota_W(c)$.
\end{proposition}

\begin{figure}
\includegraphics[scale=0.4]{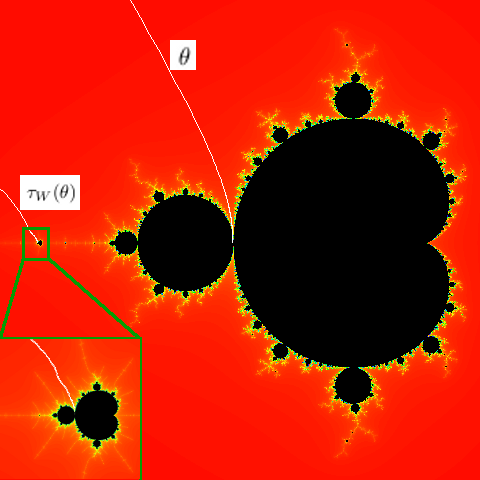}
\caption{The Mandelbrot set $\mathcal{M}$, and a baby copy of along the real axis (inside the framed rectangle).
The tuning homeomorphism maps the parameter at angle $\theta$ to the parameter at angle $\tau_W(\theta)$.}
\end{figure}

\subsection*{The real slice}

If we now restrict ourselves to the case when $c \in [-2, \frac{1}{4}]$ is real, 
each map $f_c$ acts on the real line, and has a well-defined topological entropy $h(f_c)$.
A classical result is the

\begin{theorem}[\cite{MT}, Douady-Hubbard]
The entropy of the quadratic family $h(f_c)$ is a continuous,
decreasing function of $c \in [-2, \frac{1}{4}]$.
\end{theorem}

If $W$ is a real hyperbolic component, then $\iota_W$ preserves the real axis. We will call the \emph{tuning window relative to $c_0$} the 
intersection of the real axis with the baby Mandelbrot set generated by the hyperbolic component $W$ with root $c_0$. 
Also in the case of quadratic polynomials, plateaux of the entropy are tuning windows:

\begin{theorem}[\cite{Do}]
$h(f_c) = h(f_{c'})$ if and only if $c$ and $c'$ lie in the same tuning window relative to some $c_0$, with $h(f_{c_0}) > 0$.
\end{theorem}

The set of angles $\mathcal{R}$ corresponding to external rays which land on the real slice of the Mandelbrot set
 $\partial \mathcal{M} \cap \mathbb{R}$ has a nice combinatorial description: in fact it coincides, up to a set of Hausdorff dimension zero, 
with the set
$$\mathcal{R} := \{ \theta \in \mathbb{R}/\mathbb{Z} \ :\ T^{k+1}(\theta) \leq T(\theta) \ \forall k \geq 0 \}$$
where $T(x) := \min \{ 2x, 2-2x \}$ is the usual tent map (see \cite{BCIT}, Proposition 3.4). 
A more general introduction to the combinatorics of $\mathcal{M}$ and the real slice can be found in \cite{Za}.

\subsection*{Dictionary}

By using the above combinatorial description, one can establish an isomorphism between the real slice of the boundary of the Mandelbrot set 
and the bifurcation set $\mathcal{E}$ for $\alpha$-continued fractions. Indeed, the following is true:

\begin{proposition}[\cite{BCIT}, Theorem 1.1 and Proposition 5.1]
The map $\varphi : [0, 1] \rightarrow [0,\frac{1}{2}]$
$$[0; a_1, a_2, \dots ] \mapsto 0.0\underbrace{1\dots 1}_{a_1}\underbrace{0\dots0}_{a_2}\dots$$
is a continuous bijection which maps the bifurcation set for $\alpha$-continued fractions $\mathcal{E}$ to 
the set of real rays $\mathcal{R} \cap [0, \frac{1}{2}]$.
\end{proposition}

As a corollary, 
$$
\begin{array}{c} I_r = (\alpha_1, \alpha_0)  \\
\textup{ is a maximal quadratic interval } 
\end{array}
\Leftrightarrow 
\begin{array}{c}
(\varphi(\alpha_1), \varphi(\alpha_0)) \\
\textup{is a real hyperbolic component} 
\end{array}
$$

Let us now check that the definition of tuning operators given in section \ref{sec:TW} and the Douady-Hubbard tuning correspond to 
each other via the dictionary:

\begin{proposition}
Suppose $W$ is a real hyperbolic component and let $c \in \partial \mathcal{M} \cap \mathbb{R}$ be its root. Moreover, let 
$\eta_0 \in [0, \frac{1}{2}]$ be the external angle of a ray which lands at $c$. Suppose $\eta_0$ has binary expansion 
$$\eta_0 = 0.\overline{0\underbrace{1\dots1}_{b_1}\underbrace{0\dots0}_{b_2} \dots \underbrace{0\dots0}_{b_n -1}}$$
Then, for each $x \in [0, 1]$,
$$\tau_W(\varphi(x)) = \varphi(\tau_r(x))$$
with $r = [0; b_1, \dots, b_n]$.
\end{proposition}

\begin{proof}
Suppose $x = [0; a_1, a_2, \dots]$ so that 
$\tau_r(x) = [0; S_1 S_0^{a_1-1} S_1 S_0^{a_2-1} \dots]$.
Then by definition 
$$\varphi(\tau_r(x)) = 0.01^{b_1}\dots0^{b_n-1}1 \left( 0^{b_1} \dots 1^{b_n} \right)^{a_1-1} 0^{b_1} \dots 1^{b_n-1}0\left(1^{b_1} \dots 0^{b_n} \right)^{a_2-1} \dots=$$
$$= 0.\left(01^{b_1}\dots0^{b_n-1} \right)\left( 1  0^{b_1} \dots 1^{b_n-1}\right)^{a_1} \dots = 0.\Sigma_0 \Sigma_1^{a_1}\Sigma_0^{a_2}\Sigma_1^{a_3}\dots = \tau_W(\theta)$$
where $\theta = 0.01^{a_1}0^{a_2}\dots = \varphi(x) $.
\end{proof}

Let us point out that thinking in terms of binary expansions often simplifies the 
combinatorial picture: as an example, since the monotonicity of $\tau_W$ is straightforward, the dictionary  
gives a simpler alternative proof of Lemma \ref{inj}.

\end{document}